%% file: RD_Algo_arXive.tex
\documentclass[a4paper]{article}
\usepackage{lmodern}
\pdfoutput=1
\usepackage{graphicx,wrapfig}
\usepackage{lipsum}
\usepackage{amsmath,amssymb,mathrsfs}
\usepackage{hyperref}
\usepackage{url}
\usepackage{mathtools}
\usepackage{changepage}
\usepackage[ruled,vlined]{algorithm2e}
\usepackage{amsmath}
\usepackage{multirow}
\usepackage{wrapfig}
\usepackage{subfig}
\usepackage{color}
\usepackage{epsfig,enumerate,amsmath,amsfonts,amssymb,amsthm,mathrsfs,ifpdf}
\usepackage{indentfirst,relsize}
\usepackage{setspace,graphicx}
\usepackage{latexsym}
\usepackage[all]{xy}
\usepackage[dvipsnames]{pstricks}
\usepackage{pst-grad} % For gradients
\usepackage{pst-plot} % For axes
\usepackage[margin = 2.50cm]{geometry}

\usepackage{authblk} % to add author affiliations
\usepackage{algpseudocode}
\usepackage{color,soul}
\usepackage{enumitem}% http://ctan.org/pkg/enumitem

\usepackage{orcidlink}% for orcid id

\newcommand{\remove}[1]{}

\newtheorem{theorem}{Theorem}%[section]
\newtheorem{lemma}[theorem]{Lemma}

\newtheorem{corollary}[theorem]{Corollary}

\newtheorem{remark}{Remark}

\usepackage{todonotes}
\usepackage[most]{tcolorbox}

\title{Roman domination on subclasses of bipartite graphs}
\author[1]{Gautam K. Das\footnote{gkd@iitg.ac.in}}
\author[2]{Sasmita Rout\footnote{sasmitarout20.84@gmail.com, sasmita.r@srmap.edu.in}}
\author[3]{Kamal Santra \orcidlink{0009-0006-5997-1452} \footnote{kamal.7.2013@gmail.com, kamal.santra@iitg.ac.in}}
\affil[1,3]{Department of Mathematics\\
	
	Indian Institute of Technology Guwahati\\
	
	Guwahati, 781039, Assam, India}

\affil[2]{Department of Computer Science and Engineering \\
	
	SRM University-AP\\
	
	Amaravati, 522240, Andhra Pradesh, India}
\date{}

\begin{document}

	\maketitle
	\begin{abstract}
	The Roman Domination Problem (RDP) on a simple, finite, undirected graph \(G=(V,E)\) asks for a labeling function \(f:V\rightarrow\{0,1,2\}\) such that every vertex assigned value \(0\) is adjacent to at least one vertex assigned value \(2\). The objective is to minimize the total weight \(\sum_{v\in V} f(v)\), and this minimum value is called the Roman domination number of \(G\), denoted by \(\gamma_R(G)\). Since the RDP is NP-complete for bipartite graphs, a natural direction is to study its complexity on restricted subclasses of bipartite graphs. The problem remains NP-complete even under strong structural restrictions, such as star-convex and comb-convex bipartite graphs. Therefore, identifying the borderline subclasses where the problem changes from NP-complete to polynomial-time solvable remains an important challenge.
	
	In this paper, we investigate the RDP on convex bipartite graphs and on their superclass, chordal bipartite graphs. First, we present a dynamic programming algorithm for convex bipartite graphs. The algorithm uses the interval ordering of one bipartition class and keeps a compact boundary state, which is sufficient to control the domination requirements of both processed and future vertices. This gives an \(O(n^3)\)-time algorithm for computing \(\gamma_R(G)\) on an \(n\) vertex convex bipartite graph. In contrast, we prove that the decision version of the RDP is NP-complete on chordal bipartite graphs by a polynomial reduction from \textsc{Dominating Set} on chordal bipartite graphs. Thus, our results show a clear separation between the tractability of convex bipartite graphs and the hardness of the larger chordal bipartite class.
	\end{abstract}


	%%%%%%%%%%%%%%%%%%%%%%%%%%%%%%%%%%%%%%%%%%%%%%%%%%%%%%%%%%%%%%%%%%%%%%%%%%%%%%%%%%%%%%%%%%%%%%%%%%%%%%

\input{Introduction}

\input{Preliminaries}

\input{RD_Convex}

\input{RD_Chordal_Bipartite}

\section{Conclusion}\label{sec:conclusion}
In this paper, we studied the Roman Domination Problem on two important subclasses of bipartite graphs: convex bipartite graphs and chordal bipartite graphs. For convex bipartite graphs, we gave a dynamic programming algorithm that follows the interval ordering of one bipartition class. The algorithm keeps only two pieces of boundary information: the rightmost processed vertex of \(X\) assigned value \(2\), and the earliest processed vertex of \(X\) assigned value \(0\) that is still waiting to be dominated. This compact state is sufficient because every future vertex of \(Y\) sees the already processed part of \(X\) as a suffix. Hence, the Roman domination number of an \(n\)-vertex convex bipartite graph can be computed in \(O(n^3)\) time.

We also showed that this tractability does not extend to chordal bipartite graphs. By reducing from \textsc{Dominating Set} on chordal bipartite graphs, we proved that the decision version of \textsc{Roman Domination} remains NP-complete on this class. Thus, convex bipartite graphs form a polynomially solvable subclass, while the larger class of chordal bipartite graphs is already computationally hard.

These results give a clearer picture of the complexity of Roman domination in bipartite graph classes. It would be interesting to study whether similar dynamic programming ideas can be applied to other structured bipartite graphs, such as circular-convex and triad-convex bipartite graphs.

\bibliographystyle{plain}% the mandatory bibstyle
\bibliography{RD_ref}

\end{document}

%% file: Introduction.tex
\section{Introduction}\label{sec:introduction}
Domination is one of the most widely studied topics in graph theory, both for its rich theoretical structure and its many applications in computer science. It naturally models problems in network design, facility-location planning, mobile networks, wireless sensor networks, and social network analysis. For a graph \(G=(V,E)\), a set \(S\subseteq V\) is called a dominating set if every vertex outside \(S\) has at least one neighbour in \(S\). In other words, for every \(v\in V\setminus S\), there exists a vertex \(u\in S\) such that \(uv\in E\). A dominating set of minimum cardinality is called a minimum dominating set, and its size is the domination number of \(G\), denoted by \(\gamma(G)\). The classical domination problem has been intensively studied; see, for example, \cite{haynes2020topics,haynes1998fundamentals,hedetniemi1991bibliography}. Over the years, many variants of domination have also been introduced to model different types of constraints. These include total domination \cite{cockayne1980total,jena2021total}, semi-total domination \cite{marcon2014semitotal,rout2026semi}, Roman domination \cite{cockayne2004roman,rout2025total}, Italian domination \cite{henning2017italian,haynes2020graphs}, and several others.

In this paper, we focus on Roman domination. A Roman dominating function, or RDF, on a simple undirected graph \(G=(V,E)\) is a function \(f:V\rightarrow\{0,1,2\}\) such that every vertex \(v\) with \(f(v)=0\) has at least one neighbor \(u\) with \(f(u)=2\). The weight of an RDF \(f\) is \(W(f)=\sum_{v\in V}f(v)\). If \(\mathbb{F}_G\) denotes the set of all RDFs of \(G\), then the Roman domination number of \(G\), denoted by \(\gamma_R(G)\), is defined as \(\gamma_R(G)=\min_{f\in\mathbb{F}_G} W(f)\). The Roman Domination Problem asks for an RDF of minimum weight.

The idea behind Roman domination is usually explained through a historical defense strategy. A vertex assigned value \(0\) represents an undefended location, a vertex assigned value \(1\) represents a location with one unit of defense, and a vertex assigned value \(2\) represents a location with two units of defense. The condition says that every undefended location must be adjacent to a location with two units, so that one unit can be moved there if needed. The problem was formally introduced by Cockayne et al.~\cite{cockayne2004roman}, inspired by Stewart's article on defending the Roman Empire with limited military resources~\cite{stewart1999defend}.

From the computational point of view, Roman domination is hard in general. The problem is NP-complete for general graphs \cite{dreyer2000applications}, and this hardness remains true even for several restricted graph classes, including bipartite graphs, split graphs, and planar graphs \cite{cockayne2004roman,mcRae2002private}. Because of this, a natural direction is to study the problem on special graph classes and identify where the boundary between NP-complete and polynomial-time solvable cases lies. Many structural and algorithmic results for Roman domination and its variants are known for different graph families \cite{cera2026roman,chellali2020varieties,favaron2009roman,xueliang2009roman,liu2013roman,shao2019total}.

In this work, we are particularly interested in subclasses of bipartite graphs. Let \(G=(X\cup Y,E)\) be a bipartite graph. The graph \(G\) is called tree-convex bipartite if there exists a tree \(T=(X,F)\) such that, for every vertex \(y\in Y\), the neighbourhood of \(y\) induces a subtree of \(T\). From a structural point of view, tree-convex bipartite graphs are well behaved, since they can be recognized, together with a corresponding tree representation, in linear time \cite{sheng2012review}. Several important subclasses arise by restricting the shape of the tree \(T\). If \(T\) is a star, a comb, or a path, then we obtain star-convex, comb-convex, and convex bipartite graphs, respectively.

The complexity of Roman domination varies significantly across these bipartite subclasses. The problem is NP-complete for star-convex and comb-convex bipartite graphs \cite{padamutham2020algorithmic}. On the positive side, efficient algorithms are known for several related graph classes. Linear-time algorithms have been obtained for interval graphs and cographs \cite{liedloff2008efficient}, and polynomial-time algorithms are known for bounded-treewidth graphs, chain graphs, threshold graphs, \(D\)-octopus graphs, and AT-free graphs \cite{liedloff2008efficient,padamutham2020algorithmic}. These results show that even small structural changes may lead to different computational behaviour.

The study of Roman domination on convex-type bipartite graph classes is also motivated by known results for the classical domination problem. For domination, polynomial-time algorithms are known for convex, circular-convex, and triad-convex bipartite graphs, with running times \(O(n^2)\), \(O(n^5)\), and \(O(n^8)\), respectively \cite{bang1999domination,pandey2019domination}. On the other hand, domination is NP-complete for several more general convex-type bipartite classes \cite{chen2016complexity,muller1987np}. Since Roman domination is closely related to domination, it is natural to ask whether similar tractability boundaries appear for RDP.

\subsection{Our Contribution}\label{sec:our_contribution}

Motivated by this complexity landscape, we study the Roman Domination Problem on convex bipartite graphs and chordal bipartite graphs. These two graph classes are closely related in the sense that convex bipartite graphs form a structured subclass within the broader world of chordal bipartite graphs. However, our results show that the two classes behave very differently with respect to Roman domination.

Figure~\ref{fig:F1_DifferentGraphsNPC_Status} summarizes the complexity status of RDP in several graph classes. In the figure, graph classes shown in red with solid borders are NP-complete, graph classes shown in blue with dashed borders are polynomial-time solvable, and graph classes shown in black with dotted borders are open. The classes settled in this paper are highlighted in green.

\begin{figure}[h!]
	\centering
	\includegraphics[width=0.9\textwidth]{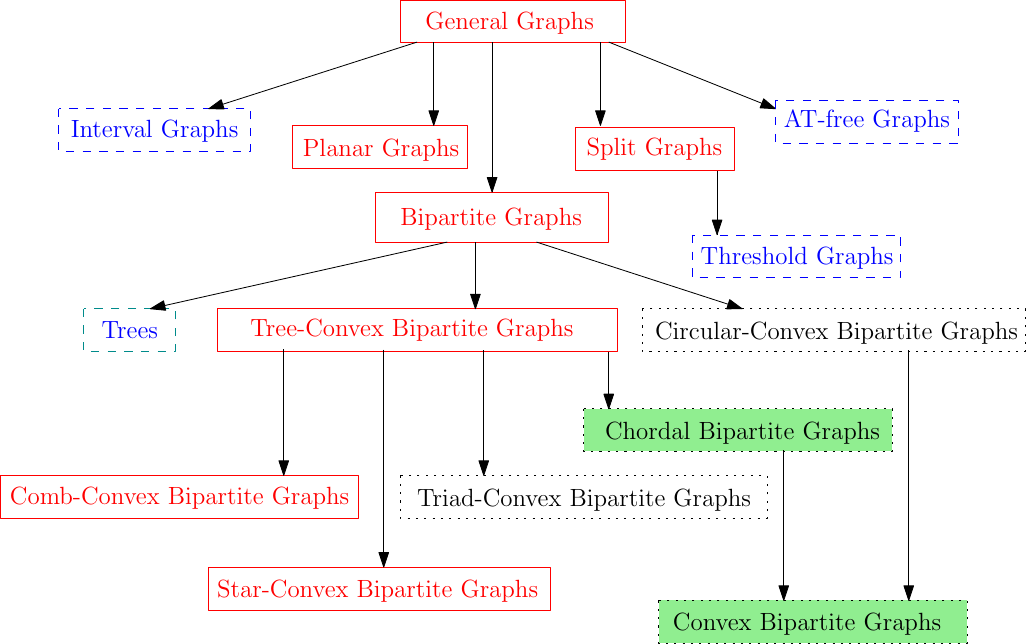}
	\caption{Complexity status of RDP in different graph classes.}
	\label{fig:F1_DifferentGraphsNPC_Status}
\end{figure}

Our first result is a polynomial-time algorithm for convex bipartite graphs. Given a convex ordering of one bipartition class \(X=\{x_1,x_2,\ldots,x_m\}\), every vertex of the other bipartition class has an interval neighbourhood in \(X\). We exploit this ordering by processing the graph from left to right. The dynamic program keeps only two boundary parameters: the rightmost processed vertex of \(X\) assigned value \(2\), and the earliest processed vertex of \(X\) assigned value \(0\) that is still waiting to be dominated. This compact state is sufficient because every future vertex of \(Y\) sees the already processed part of \(X\) as a suffix. Using this idea, we compute the Roman domination number of an \(n\)-vertex convex bipartite graph in \(O(n^3)\) time.

Our second result is a hardness result for chordal bipartite graphs. We prove that the decision version of the Roman Domination Problem remains NP-complete on chordal bipartite graphs. The proof is by a polynomial reduction from \textsc{Dominating Set} on chordal bipartite graphs. The reduction attaches one pendant vertex to each original vertex and relates the Roman domination number of the constructed graph to the domination number of the original graph.

Thus, our results give a clear separation between convex bipartite graphs and chordal bipartite graphs. Convex bipartite graphs admit an efficient dynamic programming algorithm, while the larger class of chordal bipartite graphs remains computationally hard.

The rest of the paper is organized as follows. Section~\ref{sec:preliminaries} recalls the basic definitions and notation used throughout the paper. Section~\ref{sec:RD_in_CB} presents the dynamic programming algorithm for convex bipartite graphs. Section~\ref{sec:NPC_ChordalBG} proves NP-completeness on chordal bipartite graphs. Finally, Section~\ref{sec:conclusion} concludes the paper with some directions for future research.

%% file: Preliminaries.tex
\section{Preliminaries}
\label{sec:preliminaries}

All graphs considered in this paper are finite, simple, and undirected. For a graph \(G\), we denote its vertex set and edge set by \(V(G)\) and \(E(G)\), respectively. For a vertex \(v\in V(G)\), the open neighbourhood of \(v\) is denoted by \(N_G(v)\), or simply by \(N(v)\) when the graph is clear from the context. A vertex of degree one is called a pendant vertex.

A set \(D\subseteq V(G)\) is called a dominating set of \(G\) if every vertex in \(V(G)\setminus D\) has at least one neighbour in \(D\). The domination number of \(G\), denoted by \(\gamma(G)\), is the minimum cardinality of a dominating set of \(G\). The decision version of \textsc{Dominating Set} asks, given a graph \(G\) and an integer \(k\), whether \(G\) has a dominating set of size at most \(k\).

A Roman dominating function, or RDF, of a graph \(G\) is a function \(f:V(G)\rightarrow\{0,1,2\}\) such that every vertex \(v\) with \(f(v)=0\) has a neighbor \(u\in N(v)\) with \(f(u)=2\). The weight of \(f\) is \(w(f)=\sum_{v\in V(G)} f(v)\). The Roman domination number of \(G\), denoted by \(\gamma_R(G)\), is the minimum weight of an RDF of \(G\). The decision version of \textsc{Roman Domination} asks, given a graph \(G\) and an integer \(k\), whether \(\gamma_R(G)\leq k\). Since the Roman domination number is additive over connected components, it is enough for the algorithmic part to describe the method for connected graphs.

Let \(G=(X\cup Y,E)\) be a bipartite graph with bipartition \((X,Y)\). We say that \(G\) is convex with respect to \(X\) if there is an ordering \(X=\{x_1,x_2,\ldots,x_m\}\) such that, for every vertex \(y\in Y\), the neighbourhood of \(y\) is consecutive in this ordering. Thus, for every \(y\in Y\), there are integers \(l(y)\) and \(r(y)\), with \(1\leq l(y)\leq r(y)\leq m\), such that \(N(y)=\{x_{l(y)},x_{l(y)+1},\ldots,x_{r(y)}\}\). We call \(l(y)\) and \(r(y)\) the left and right endpoints of \(y\), respectively. Throughout the algorithmic part, a convex bipartite graph is assumed to be given together with such a convex ordering and the endpoints \(l(y)\) and \(r(y)\) for all \(y\in Y\).

A bipartite graph is called a chordal bipartite graph if it has no induced cycle of length at least six. Equivalently, every induced cycle in a chordal bipartite graph has length four. This class is used in the hardness part of the paper.

%% file: RD_Convex.tex
\section{Roman Domination in Convex Bipartite Graphs}\label{sec:RD_in_CB}

In this section, we present a dynamic programming algorithm for the Roman Domination Problem on convex bipartite graphs. Let \(G=(X\cup Y,E)\) be a nontrivial connected bipartite graph, where \(X=\{x_1,x_2,\ldots,x_m\}\) and \(Y=\{y_1,y_2,\ldots,y_q\}\). We assume that \(G\) is convex with respect to the ordering \(x_1,x_2,\ldots,x_m\) of \(X\). Thus, for every vertex \(y\in Y\), the neighbourhood of \(y\) is consecutive in this ordering. We write \(N(y)=\{x_{l(y)},x_{l(y)+1},\ldots,x_{r(y)}\}\), where \(l(y)\) and \(r(y)\) are the left and right endpoints of the interval corresponding to \(y\). Since \(G\) is connected and nontrivial, no vertex of \(Y\) is isolated.

The algorithm scans the graph from left to right according to the ordering of \(X\). Whenever all neighbours of a vertex \(y\in Y\) have been processed, the algorithm decides whether \(y\) is already properly dominated. On the other hand, a processed vertex of \(X\) assigned value \(0\) may still be dominated later by some unprocessed vertex of \(Y\) assigned value \(2\). The dynamic program, therefore, stores only the boundary information needed to remember such unfinished vertices.

\subsection{Processing Order and States}

For each \(i\in\{1,\ldots,m\}\), let \(Y_i=\{y\in Y:r(y)=i\}\). Thus, \(Y_i\) contains exactly those vertices of \(Y\) whose interval neighbourhoods end at \(x_i\). The algorithm processes the graph in the order \(x_1,Y_1,x_2,Y_2,\ldots,x_m,Y_m\). If some \(Y_i\) is empty, then no vertex of \(Y\) is processed immediately after \(x_i\).

This ordering is useful because, when a vertex \(y\in Y_i\) is processed, all its neighbours have already appeared. Indeed, since \(r(y)=i\), we have \(N(y)=\{x_{l(y)},x_{l(y)+1},\ldots,x_i\}\). Hence, if \(y\) is assigned value \(0\), the algorithm can immediately check whether \(y\) has a neighbour of value \(2\) among the processed vertices of \(X\). In contrast, when a vertex \(x_i\) is processed, it may still be dominated later by a vertex \(y\in Y_j\) with \(j\geq i\). Therefore, the dynamic program must keep track of those processed vertices of \(X\) that currently have value \(0\) but are still waiting to be dominated by future vertices of \(Y\).

For \(i\in\{1,\ldots,m\}\), define \(X_i=\{x_1,\ldots,x_i\}\) and \(Y_{\leq i}=\{y\in Y:r(y)\leq i\}\). After the algorithm has processed \(X_i\cup Y_{\leq i}\), we call a partial assignment feasible if every vertex \(y\in Y_{\leq i}\) assigned value \(0\) is already adjacent to some vertex of \(X_i\) assigned value \(2\). We do not require the vertices of \(X_i\) assigned value \(0\) to be already dominated, because such vertices may still be dominated by vertices of \(Y\) that will be processed later.

Let \(f\) be a feasible partial assignment on \(X_i\cup Y_{\leq i}\). We define \(s(f)\) to be the largest index \(t\leq i\) such that \(f(x_t)=2\). If no processed vertex of \(X\) has value \(2\), then \(s(f)=0\). Thus, \(s(f)\) records the rightmost processed vertex of \(X\) assigned value \(2\).

Next, define the pending set \(P(f)\). An index \(t\leq i\) belongs to \(P(f)\) if \(f(x_t)=0\) and \(x_t\) has no neighbour in \(Y_{\leq i}\) assigned value \(2\). Thus, \(P(f)\) consists exactly of the processed vertices of \(X\) that currently have value \(0\) but are not yet dominated by any processed vertex of \(Y\). These vertices are allowed to remain temporarily undominated since they may still be dominated by future vertices of \(Y\). If \(P(f)\neq\emptyset\), let \(p(f)=\min P(f)\); otherwise, let \(p(f)=p_\infty\), where \(p_\infty=m+1\). Hence, \(p(f)\) records the earliest pending vertex of \(X\), and \(p_\infty\) means that no pending vertex exists.

The state of the dynamic program is the pair \((s,p)\). The table stores, for each state \((s,p)\), the minimum weight of a feasible partial assignment having that state. Notice that \(p_\infty\) is a state value, whereas \(+\infty\) is used only as a table value for unreachable states.

We now justify why only the earliest pending vertex is needed. Consider any future vertex \(y\in Y\setminus Y_{\leq i}\). Since \(G\) is convex with respect to the ordering of \(X\), the set \(N(y)\cap X_i\) is either empty or a suffix of \(X_i\). More precisely, if \(l(y)\leq i\), then \(N(y)\cap X_i=\{x_{l(y)},x_{l(y)+1},\ldots,x_i\}\); otherwise, \(N(y)\cap X_i=\emptyset\). Therefore, if a future vertex \(y\) assigned value \(2\) dominates the earliest pending vertex \(x_p\), then it also dominates every pending vertex with index larger than \(p\). If it does not dominate \(x_p\), then \(x_p\) remains pending. Thus, for all future decisions, the whole pending set is represented by its minimum index \(p\).

Before the first vertex is processed, the table contains only the empty assignment. This assignment has weight \(0\), no processed vertex of \(X\) assigned value \(2\), and no pending vertex. Therefore, we initialize the table by setting \(D(0,p_\infty)=0\), and all other entries are set to \(+\infty\).

\subsection{Transition Rules}

Suppose the table has been computed after processing \(X_{i-1}\cup Y_{\leq i-1}\). We first process the next vertex \(x_i\). For every reachable state \((s,p)\) with \(D(s,p)<+\infty\), we consider the three possible values of \(x_i\).

If \(f(x_i)=0\), then \(x_i\) becomes pending. Indeed, every already processed vertex of \(Y\) has right endpoint at most \(i-1\), and hence no processed vertex of \(Y\) is adjacent to \(x_i\). Thus, the new earliest pending index is \(p_0=\min\{p,i\}\), and the weight does not increase.

If \(f(x_i)=1\), then \(x_i\) is safe by itself. It does not become pending, and it does not change the rightmost processed vertex of \(X\) assigned value \(2\). Hence, the state remains \((s,p)\), and the weight increases by \(1\).

If \(f(x_i)=2\), then \(x_i\) is safe and becomes the rightmost processed vertex of \(X\) assigned value \(2\). Therefore, the new state is \((i,p)\), and the weight increases by \(2\).

After all transitions for \(x_i\) have been performed, we process the vertices in \(Y_i\) one by one. Let \(y\in Y_i\), and let \(l=l(y)\). Since \(r(y)=i\), all neighbours of \(y\) have already been processed, and \(N(y)=\{x_l,x_{l+1},\ldots,x_i\}\). For every reachable state \((s,p)\) with \(D(s,p)<+\infty\), we again consider the three possible values of \(y\).

If \(f(y)=0\), then \(y\) must already be dominated by a processed vertex of \(X\) assigned value \(2\). Since \(s\) is the largest index of a processed vertex of \(X\) assigned value \(2\), such a neighbour exists in \(N(y)\) if and only if \(s\geq l\). If \(s\geq l\), then \(x_s\in N(y)\) and \(f(x_s)=2\), so the choice \(f(y)=0\) is valid and the state does not change. If \(s<l\), then no processed vertex of \(X\) assigned value \(2\) lies in \(N(y)\), and this choice is infeasible.

If \(f(y)=1\), then \(y\) is safe by itself. A vertex assigned value \(1\) does not dominate pending vertices of \(X\) in the Roman domination condition. Therefore, the state remains \((s,p)\), and the weight increases by \(1\).

If \(f(y)=2\), then \(y\) dominates every pending vertex of \(X\) that lies in its interval \(N(y)=\{x_l,\ldots,x_i\}\). If \(p\neq p_\infty\) and \(p\geq l\), then the earliest pending vertex \(x_p\) lies in \(N(y)\). Every other pending vertex has index at least \(p\), and at most \(i\), and hence it also lies in \(N(y)\). Therefore, all pending vertices are dominated, and the new pending value becomes \(p_\infty\). If \(p=p_\infty\), then there is no pending vertex. If \(p<l\), then the earliest pending vertex \(x_p\) is not adjacent to \(y\), and hence it remains pending. Thus, if \(p\neq p_\infty\) and \(p\geq l\), we set \(p_2=p_\infty\); otherwise, we set \(p_2=p\). The new state is \((s,p_2)\), and the weight increases by \(2\).

\subsection{Algorithm}

\begin{algorithm}[H]
	\small
	\DontPrintSemicolon
	\caption{\textsc{Roman-Convex-Bipartite}}
	\KwIn{A connected convex bipartite graph \(G=(X\cup Y,E)\), where \(X=\{x_1,\ldots,x_m\}\), together with the interval representation \(N(y)=\{x_{l(y)},\ldots,x_{r(y)}\}\) for every \(y\in Y\).}
	\KwOut{The Roman domination number \(\gamma_R(G)\).}
	
	Let \(p_\infty=m+1\)\;
	\For{\(i=1\) \KwTo \(m\)}{
		Initialize \(Y_i\leftarrow\emptyset\)\;
	}
	\ForEach{\(y\in Y\)}{
		Append \(y\) to \(Y_{r(y)}\)\;
	}
	
	Set \(D(0,p_\infty)=0\), and set all other entries to \(+\infty\)\;
	
	\For{\(i=1\) \KwTo \(m\)}{
		Set all entries of \(D'\) to \(+\infty\)\;
		
		\ForEach{state \((s,p)\) with \(D(s,p)<+\infty\)}{
			\(p_0=\min\{p,i\}\)\;
			update \(D'(s,p_0)\) with \(D(s,p)\)\tcp*{\(f(x_i)=0\)}
			
			update \(D'(s,p)\) with \(D(s,p)+1\)\tcp*{\(f(x_i)=1\)}
			
			update \(D'(i,p)\) with \(D(s,p)+2\)\tcp*{\(f(x_i)=2\)}
		}
		
		Set \(D=D'\)\;
		
		\ForEach{\(y\in Y_i\)}{
			Set all entries of \(D'\) to \(+\infty\)\;
			Let \(l=l(y)\)\;
			
			\ForEach{state \((s,p)\) with \(D(s,p)<+\infty\)}{
				\If{\(s\geq l\)}{
					update \(D'(s,p)\) with \(D(s,p)\)\tcp*{\(f(y)=0\)}
				}
				
				update \(D'(s,p)\) with \(D(s,p)+1\)\tcp*{\(f(y)=1\)}
				
				\eIf{\(p\neq p_\infty\) and \(p\geq l\)}{
					\(p_2=p_\infty\)\;
				}{
					\(p_2=p\)\;
				}
				
				update \(D'(s,p_2)\) with \(D(s,p)+2\)\tcp*{\(f(y)=2\)}
			}
			
			Set \(D=D'\)\;
		}
	}
	
	\Return{\(\min\{D(s,p_\infty):0\leq s\leq m\}\)}\;
\end{algorithm}

The pseudocode above uses the following table convention. The table \(D\) stores the best values for the states reached after the part of the graph processed so far. A state \((s,p)\) is considered only if it is reachable at the current stage, that is, only if \(D(s,p)<+\infty\). Such a state corresponds to at least one feasible partial assignment. States with value \(+\infty\) are infeasible at that stage and are ignored.

For each processed vertex, the algorithm creates a temporary table \(D'\), whose entries are initially set to \(+\infty\). All feasible transitions from the current table \(D\) are then written into \(D'\). In the pseudocode, the phrase ``update \(D'(a,b)\) with \(w\)'' means \(D'(a,b)=\min\{D'(a,b),w\}\). Once all transitions for the current vertex have been completed, \(D'\) replaces \(D\) as the current table.

\subsection{Correctness Proof}

We now prove the correctness of Algorithm \textsc{Roman-Convex-Bipartite}. Recall that a state \((s,p)\) summarizes the boundary information needed for the remaining computation. Here \(s\) is the rightmost processed vertex of \(X\) assigned value \(2\), and \(p\) is the earliest processed vertex of \(X\) that is still pending. If no such pending vertex exists, then \(p=p_\infty\).

A state with \(p\neq p_\infty\) is allowed during the intermediate stages of the algorithm. Such a state means that the vertex \(x_p\) is currently undominated but may still be dominated later by some unprocessed vertex of \(Y\) assigned value \(2\). However, after all vertices have been processed, no future vertex remains. Hence, a final state with \(p\neq p_\infty\) cannot represent a Roman dominating function. Therefore, at the end of the algorithm, only states of the form \((s,p_\infty)\) are accepted.

\begin{lemma}\label{RD_convex_correctness_lemma}
	After each processing step, for every state \((s,p)\), the current table entry \(D(s,p)\) has the following meaning. If there exists a feasible partial assignment with state \((s,p)\), then \(D(s,p)\) is the minimum weight among all such assignments. If no feasible partial assignment realizes \((s,p)\), then \(D(s,p)=+\infty\).
\end{lemma}

\begin{proof}
	We prove the lemma by induction on the processing order.
	
	Before any vertex is processed, the only partial assignment is the empty assignment. Its weight is \(0\). No vertex of \(X\) has been processed, so no processed vertex of \(X\) is assigned value \(2\). Also, there is no pending vertex. Hence, the initial state is \((0,p_\infty)\), and the initialization \(D(0,p_\infty)=0\) is correct. Every other state is impossible at this stage and is correctly initialized to \(+\infty\).
	
	Assume that the table is correct just before a new vertex is processed. We show that the transition rules produce the correct table after processing that vertex. Since the temporary table \(D'\) is first initialized with value \(+\infty\) in every entry, a state receives a finite value only if it is produced by a valid transition from a reachable state of the current table.
	
	First, suppose that the next vertex is \(x_i\). Let \((s,p)\) be a reachable state in the current table. By the induction hypothesis, \(D(s,p)\) is the minimum weight of a feasible partial assignment with state \((s,p)\). Since \(x_i\) can receive only one of the values \(0,1,2\), we consider these three choices separately.
	
	If \(f(x_i)=0\), then \(x_i\) becomes pending. Indeed, every vertex of \(Y\) processed before \(x_i\) has right endpoint at most \(i-1\), and hence none of them is adjacent to \(x_i\). Therefore, \(x_i\) is not dominated by any processed vertex of \(Y\) assigned value \(2\). The earliest pending index becomes \(\min\{p,i\}\), while the first coordinate \(s\) remains unchanged. The weight does not increase. Thus, the transition from \((s,p)\) to \((s,\min\{p,i\})\) is correct.
	
	If \(f(x_i)=1\), then \(x_i\) is safe by itself. It does not become pending, and it does not create a new vertex of \(X\) assigned value \(2\). Hence the state remains \((s,p)\), and the weight increases by \(1\). Thus, the transition from \((s,p)\) to \((s,p)\) with additional cost \(1\) is correct.
	
	If \(f(x_i)=2\), then \(x_i\) is safe and becomes the rightmost processed vertex of \(X\) assigned value \(2\). Hence, the first coordinate becomes \(i\). The pending coordinate remains \(p\), because assigning value \(2\) to a vertex of \(X\) does not dominate pending vertices of \(X\); there are no edges inside \(X\). The weight increases by \(2\). Thus, the transition from \((s,p)\) to \((i,p)\) with additional cost \(2\) is correct.
	
	These three transitions are exhaustive for \(x_i\), since no other value is possible. Every retained transition extends a feasible partial assignment to another feasible partial assignment, and every feasible extension to include \(x_i\) is represented by exactly one of these three transitions. Therefore, no feasible possibility is lost while processing \(x_i\), and the algorithm keeps the minimum weight whenever several transitions reach the same state.
	
	Now suppose that the next vertex is \(y\in Y_i\). Since \(r(y)=i\), all neighbours of \(y\) have already been processed, and \(N(y)=\{x_{l(y)},x_{l(y)+1},\ldots,x_i\}\). Let \(l=l(y)\). Again, let \((s,p)\) be a reachable state in the current table. There are exactly three possible values for \(y\).
	
	If \(f(y)=0\), then \(y\) must have a neighbour assigned value \(2\). Since all neighbours of \(y\) lie in \(X\), this means that some vertex in the interval \(N(y)=\{x_l,\ldots,x_i\}\) must be assigned value \(2\). The state records the rightmost processed vertex of \(X\) assigned value \(2\), namely \(x_s\). Therefore, such a neighbour exists if and only if \(s\geq l\). If \(s\geq l\), then assigning \(f(y)=0\) is feasible, the state remains \((s,p)\), and the weight does not increase. If \(s<l\), then no processed vertex of \(X\) assigned value \(2\) lies in \(N(y)\). Since all neighbours of \(y\) have already been processed, no future vertex can later dominate \(y\). Hence, this transition is infeasible and is correctly discarded.
	
	If \(f(y)=1\), then \(y\) is safe by itself. A vertex assigned value \(1\) does not need to be dominated, and it does not dominate any pending vertex of \(X\). Hence, the state remains \((s,p)\), and the weight increases by \(1\). This transition is always feasible.
	
	Finally, suppose that \(f(y)=2\). Then \(y\) dominates exactly those pending vertices of \(X\) that lie in its interval \(N(y)=\{x_l,\ldots,x_i\}\). If \(p=p_\infty\), then there is no pending vertex, so the pending coordinate remains \(p_\infty\). If \(p\neq p_\infty\) and \(p\geq l\), then the earliest pending vertex \(x_p\) lies in \(N(y)\). Every later pending vertex has index at least \(p\), and every processed vertex of \(X\) has index at most \(i\). Therefore, every pending vertex lies in \(N(y)\), and all pending vertices become dominated. In this case, the new pending value becomes \(p_\infty\). If \(p<l\), then \(x_p\notin N(y)\). Hence, the earliest pending vertex remains undominated, and the pending coordinate remains \(p\). Although \(y\) may dominate some later pending vertices, this does not change the state, because the state records only the earliest pending vertex. Thus, the transition for \(f(y)=2\) is correct.
	
	The three transitions for \(y\) are exhaustive, since \(y\) can receive only one of the values \(0,1,2\). Every retained transition is valid, and every feasible choice for \(y\) is represented by one of these transitions. Therefore, no feasible possibility is lost while processing \(y\), and the algorithm keeps the minimum weight whenever several transitions reach the same state.
	
	It remains to justify that the state \((s,p)\) contains enough information for all future decisions. The value \(s\) is sufficient for deciding whether a processed vertex \(y\in Y_i\) can be assigned value \(0\). Indeed, \(y\) has a neighbour of value \(2\) in \(X\) exactly when the rightmost processed value-\(2\) vertex of \(X\) has index at least \(l(y)\). This is precisely the condition \(s\geq l(y)\). Thus, no other information about the set of value-\(2\) vertices in \(X\) is needed for checking vertices of \(Y\).
	
	The value \(p\) is sufficient for representing all pending vertices of \(X\). Let \(y\) be any future vertex of \(Y\). Since \(G\) is convex with respect to the ordering of \(X\), the set \(N(y)\cap X_i\) is either empty or a suffix of \(X_i\). Therefore, if \(y\) is assigned value \(2\) and dominates the earliest pending vertex \(x_p\), then it also dominates every later pending vertex. If \(y\) does not dominate \(x_p\), then \(x_p\) remains pending, regardless of what happens to later pending vertices. Hence, the future behaviour of the whole pending set is determined by its minimum index \(p\).
	
	For a fixed state \((s,p)\), there may be many feasible partial assignments realizing that state. The algorithm keeps only the minimum weight among them. This is safe because all future transitions depend only on \((s,p)\), not on the internal details of the partial assignment. Therefore, a heavier partial assignment with the same state can never lead to a better final solution than a lighter one with that state.
	
	Combining the above arguments, every table entry produced by the algorithm has the correct value, and every feasible partial assignment after the current step is represented in the new table. States that cannot be realized remain equal to \(+\infty\). Thus, the invariant holds after the processing step. The lemma follows by induction.
\end{proof}

\begin{theorem}\label{thm:roman_convex_correct}
	Algorithm \textsc{Roman-Convex-Bipartite} computes \(\gamma_R(G)\).
\end{theorem}

\begin{proof}
	After all vertices have been processed, every vertex of \(Y\) assigned value \(0\) has already been checked at the moment it was processed. Indeed, when a vertex \(y\in Y_i\) is processed, all its neighbours lie in \(X_i\), and hence all of them have already been processed. Therefore, in every feasible final assignment, all vertices of \(Y\) satisfy the Roman domination condition.
	
	The only possible obstruction is a pending vertex of \(X\). If the final state is \((s,p)\) with \(p\neq p_\infty\), then, by the definition of the state, every assignment represented by this state has the vertex \(x_p\) still undominated. Since no future vertex remains, \(x_p\) can no longer be dominated. Hence, no final state with \(p\neq p_\infty\) can represent a Roman dominating function.
	
	On the other hand, if the final state is \((s,p_\infty)\), then the pending set is empty. Thus, no vertex of \(X\) remains undominated. There is no hidden pending vertex in such a state, because \(p=p_\infty\) is used only when the pending set is empty. Moreover, whenever a vertex of \(Y\) assigned value \(2\) dominates the earliest pending vertex, it also dominates every later pending vertex, since its neighbourhood in the processed part of \(X\) is a suffix. Therefore, a final state of the form \((s,p_\infty)\) represents a Roman dominating function of \(G\).
	
	At least one accepted final state always exists. The assignment that gives value \(1\) to every vertex of \(G\) is a Roman dominating function, because no vertex is assigned value \(0\). This assignment is considered by the algorithm by choosing \(f(x_i)=1\) for every \(x_i\in X\) and \(f(y)=1\) for every \(y\in Y\). It never creates a pending vertex and ends in the state \((0,p_\infty)\) with weight \(|V(G)|\). Hence, the set of final states of the form \((s,p_\infty)\) with finite table value is nonempty.
	
	By Lemma~\ref{RD_convex_correctness_lemma}, for each \(s\in\{0,1,\ldots,m\}\), the entry \(D(s,p_\infty)\) stores the minimum weight of a feasible final assignment with no pending vertex and with \(s\) as the rightmost processed vertex of \(X\) assigned value \(2\). Therefore, the minimum weight of a Roman dominating function of \(G\) is \(\min\{D(s,p_\infty):0\leq s\leq m\}\). This is exactly the value returned by Algorithm \textsc{Roman-Convex-Bipartite}. Hence, the algorithm computes \(\gamma_R(G)\).
\end{proof}

\subsection{Running Time Analysis}

\begin{theorem}\label{thm:roman_convex_runtime}
	Let \(G=(X\cup Y,E)\) be a connected convex bipartite graph on \(n\) vertices, given together with a convex ordering of \(X\) and the interval endpoints \(l(y)\) and \(r(y)\) for every vertex \(y\in Y\). Then Algorithm \textsc{Roman-Convex-Bipartite} computes \(\gamma_R(G)\) in \(O(n^3)\).
\end{theorem}

\begin{proof}
	Let \(m=|X|\) and \(q=|Y|\). Then \(n=m+q\). We first count the number of possible states. After the algorithm has processed up to \(x_i\), the first coordinate \(s\) belongs to \(\{0,1,\ldots,i\}\), where \(s=0\) means that no processed vertex of \(X\) has value \(2\). The second coordinate \(p\) belongs to \(\{1,\ldots,i\}\cup\{p_\infty\}\), where \(p=p_\infty\) means that there is no pending vertex. Hence, for each fixed \(i\), the number of possible states is at most \((i+1)^2\), which is \(O(m^2)\).
	
	Each vertex \(x_i\in X\) is processed once. For every reachable state \((s,p)\), the algorithm tries three choices, corresponding to \(f(x_i)=0\), \(f(x_i)=1\), and \(f(x_i)=2\). Each transition updates only a constant number of table entries and therefore takes constant time. Since there are \(m\) vertices in \(X\) and \(O(m^2)\) possible states at each step, the total time spent on the vertices of \(X\) is \(O(m^3)\).
	
	Each vertex \(y\in Y\) is also processed once. For every reachable state \((s,p)\), the algorithm considers at most three choices, corresponding to \(f(y)=0\), \(f(y)=1\), and \(f(y)=2\). The transition for \(f(y)=0\) only checks whether \(s\geq l(y)\). The transition for \(f(y)=2\) only checks whether \(p\neq p_\infty\) and \(p\geq l(y)\). Thus, each transition takes constant time. Since each vertex \(y\in Y\) is processed once and there are \(O(m^2)\) possible states, the total time spent on the vertices of \(Y\) is \(O(qm^2)\).
	
	It remains only to group the vertices of \(Y\) into the sets \(Y_i\). If the interval endpoints are already given, this grouping can be done in \(O(m+q)\) time by first initializing all lists \(Y_i\) as empty and then placing each vertex \(y\) into the list indexed by \(r(y)\). If the endpoints are computed from adjacency lists, then this preprocessing takes \(O(|E|)\) time. Since \(|E|\leq mq\leq n^2\), this preprocessing does not dominate the dynamic programming time.
	
	Therefore, the total running time is \(O(m^3+qm^2+|E|)\). Since \(m^3+qm^2=(m+q)m^2\), this is \(O((m+q)m^2+|E|)\). Now \(m+q=n\), \(m\leq n\), and \(|E|\leq n^2\). Hence, the total running time is \(O(n^3)\).
\end{proof}

\begin{remark}
	If an optimal Roman dominating function is required, the algorithm can store a predecessor pointer whenever a table entry is improved. Starting from a final state \((s,p_\infty)\) attaining the returned minimum, one can then backtrack through the stored choices and recover an optimal Roman dominating function. This does not change the asymptotic running time.
\end{remark}

\begin{corollary}\label{cor:roman_convex_all}
	The Roman domination number of an \(n\)-vertex convex bipartite graph, not necessarily connected, can be computed in \(O(n^3)\) time.
\end{corollary}

\begin{proof}
	The Roman domination number is additive over connected components. Therefore, we apply Algorithm \textsc{Roman-Convex-Bipartite} independently to each nontrivial connected component, using the inherited convex ordering. Each isolated vertex contributes \(1\) to the Roman domination number. If the component sizes are \(n_1,n_2,\ldots,n_t\), then the total running time is at most \(\sum_{j=1}^{t}O(n_j^3)\leq O(n^3)\). Hence, the whole graph can be handled in \(O(n^3)\) time.
\end{proof}

\subsection{Example}

We illustrate the algorithm on a small convex bipartite graph. Let \(X=\{x_1,x_2,x_3,x_4\}\) and \(Y=\{y_1,y_2,y_3\}\). Suppose that \(N(y_1)=\{x_1,x_2,x_3\}\), \(N(y_2)=\{x_2,x_3,x_4\}\), and \(N(y_3)=\{x_3,x_4\}\). Then \(G\) is convex with respect to the ordering \(x_1,x_2,x_3,x_4\) of \(X\). The right endpoints are \(r(y_1)=3\), \(r(y_2)=4\), and \(r(y_3)=4\). Hence \(Y_1=Y_2=\emptyset\), \(Y_3=\{y_1\}\), and \(Y_4=\{y_2,y_3\}\). Since \(m=4\), we have \(p_\infty=m+1=5\).

\begin{figure}[h]
	\centering
	\begin{tikzpicture}[
		xnode/.style={circle, draw, minimum size=7mm, inner sep=0pt},
		ynode/.style={circle, draw, minimum size=7mm, inner sep=0pt},
		opt2/.style={circle, draw, fill=gray!30, very thick, minimum size=7mm, inner sep=0pt},
		every node/.style={font=\small}
		]
		
		\node[xnode] (x1) at (0,2) {\(x_1\)};
		\node[xnode] (x2) at (2,2) {\(x_2\)};
		\node[xnode] (x3) at (4,2) {\(x_3\)};
		\node[opt2] (x4) at (6,2) {\(x_4\)};
		
		\node[opt2] (y1) at (1,0) {\(y_1\)};
		\node[ynode] (y2) at (3,0) {\(y_2\)};
		\node[ynode] (y3) at (5,0) {\(y_3\)};
		
		\draw (x1)--(y1);
		\draw (x2)--(y1);
		\draw (x3)--(y1);
		
		\draw (x2)--(y2);
		\draw (x3)--(y2);
		\draw (x4)--(y2);
		
		\draw (x3)--(y3);
		\draw (x4)--(y3);
		
		\node at (-0.7,2) {\(X\)};
		\node at (0.3,0) {\(Y\)};
		
		\node[above] at (x1.north) {\(0\)};
		\node[above] at (x2.north) {\(0\)};
		\node[above] at (x3.north) {\(0\)};
		\node[above] at (x4.north) {\(2\)};
		
		\node[below] at (y1.south) {\(2\)};
		\node[below] at (y2.south) {\(0\)};
		\node[below] at (y3.south) {\(0\)};
		
	\end{tikzpicture}
	\caption{A convex bipartite graph used to illustrate the algorithm. The shaded vertices are assigned value \(2\) in an optimal Roman dominating function: \(f(y_1)=2\) and \(f(x_4)=2\). All other vertices receive value \(0\).}
	\label{fig:roman-convex-example}
\end{figure}
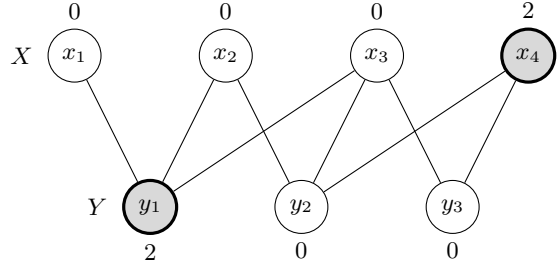

The algorithm processes the vertices in the order \(x_1,x_2,x_3,y_1,x_4,y_2,y_3\). Recall that a state is written as \((s,p)\), where \(s\) is the rightmost processed vertex of \(X\) assigned value \(2\), and \(p\) is the earliest pending vertex of \(X\). The value stored in a table entry is the minimum weight needed to reach that state.

Before each vertex is processed, the current table is denoted by \(D\). The algorithm then initializes a temporary table \(D'\), writes all valid transitions into \(D'\), and finally replaces \(D\) by \(D'\). The following table follows one optimal branch and explicitly shows the temporary entry written in \(D'\). After each processing step, the entry shown in \(D'\) becomes the corresponding entry of \(D\) for the next step.

\begin{center}
	\small
	\begin{tabular}{c|c|c|c}
		Step & Current entry in \(D\) & Chosen value & Entry written in \(D'\) \\
		\hline
		Initial & \(D(0,p_\infty)=0\) & -- & -- \\
		\hline
		Processing \(x_1\) & \(D(0,p_\infty)=0\) & \(f(x_1)=0\) & \(D'(0,1)=0\) \\
		Processing \(x_2\) & \(D(0,1)=0\) & \(f(x_2)=0\) & \(D'(0,1)=0\) \\
		Processing \(x_3\) & \(D(0,1)=0\) & \(f(x_3)=0\) & \(D'(0,1)=0\) \\
		Processing \(y_1\) & \(D(0,1)=0\) & \(f(y_1)=2\) & \(D'(0,p_\infty)=2\) \\
		Processing \(x_4\) & \(D(0,p_\infty)=2\) & \(f(x_4)=2\) & \(D'(4,p_\infty)=4\) \\
		Processing \(y_2\) & \(D(4,p_\infty)=4\) & \(f(y_2)=0\) & \(D'(4,p_\infty)=4\) \\
		Processing \(y_3\) & \(D(4,p_\infty)=4\) & \(f(y_3)=0\) & \(D'(4,p_\infty)=4\)
	\end{tabular}
\end{center}

Let us explain the entries in the table. Initially, the table contains only \(D(0,p_\infty)=0\). If we assign \(f(x_1)=0\), then \(x_1\) becomes pending. Hence, the temporary table receives the entry \(D'(0,1)=0\). After all transitions for \(x_1\) are completed, the algorithm sets \(D=D'\), so the current table contains \(D(0,1)=0\).

The same happens when \(x_2\) and \(x_3\) are assigned value \(0\). The earliest pending vertex remains \(x_1\), so the branch keeps the state \((0,1)\) with value \(0\). Thus, after processing \(x_3\), the branch has \(D(0,1)=0\).

Next, the algorithm processes \(y_1\). Since \(N(y_1)=\{x_1,x_2,x_3\}\), assigning \(f(y_1)=2\) dominates all pending vertices. Therefore, the pending coordinate changes from \(1\) to \(p_\infty\). The temporary table receives \(D'(0,p_\infty)=D(0,1)+2=2\). After copying, the current table contains \(D(0,p_\infty)=2\).

The next vertex is \(x_4\). Assigning \(f(x_4)=2\) makes \(x_4\) the rightmost processed vertex of \(X\) assigned value \(2\). Thus, the first coordinate changes from \(0\) to \(4\), the pending coordinate remains \(p_\infty\), and the temporary table receives \(D'(4,p_\infty)=D(0,p_\infty)+2=4\). After copying, the current table contains \(D(4,p_\infty)=4\).

Now \(y_2\) is processed. Since \(l(y_2)=2\) and \(s=4\), we have \(s\geq l(y_2)\). Hence, \(y_2\) is adjacent to a processed vertex of \(X\) assigned value \(2\), namely \(x_4\). Therefore, assigning \(f(y_2)=0\) is allowed, and the temporary table receives \(D'(4,p_\infty)=4\). The same argument applies to \(y_3\), since \(l(y_3)=3\) and \(s=4\geq 3\). Thus, assigning \(f(y_3)=0\) is also allowed, and the final accepted value on this branch is \(4\).

The table above follows only one optimal branch. The full temporary table \(D'\) may contain several finite entries after each processed vertex. The following table lists representative finite entries of \(D'\) immediately before the assignment \(D=D'\) is made. These entries show how the algorithm keeps both complete and incomplete partial assignments.

\begin{center}
	\small
	\begin{tabular}{p{0.22\linewidth}|p{0.68\linewidth}}
		After processing & Some finite entries of the temporary table \(D'\) before setting \(D=D'\) \\
		\hline
		\(x_1\) & \(D'(0,1)=0\), \(D'(0,p_\infty)=1\), \(D'(1,p_\infty)=2\) \\
		\(x_2\) & \(D'(0,1)=0\), \(D'(0,2)=1\), \(D'(1,2)=2\), \(D'(2,1)=2\) \\
		\(x_3\) & \(D'(0,1)=0\), \(D'(0,2)=1\), \(D'(0,3)=2\), \(D'(3,1)=2\) \\
		\(y_1\) & \(D'(0,p_\infty)=2\), \(D'(0,1)=1\), \(D'(3,1)=2\), \(D'(1,2)=2\) \\
		\(x_4\) & \(D'(4,p_\infty)=4\), \(D'(4,1)=3\), \(D'(0,1)=1\), \(D'(0,4)=2\) \\
		\(y_2\) & \(D'(4,p_\infty)=4\), \(D'(4,1)=3\), \(D'(2,1)=2\), \(D'(3,1)=2\) \\
		\(y_3\) & \(D'(4,p_\infty)=4\), \(D'(4,1)=3\), \(D'(3,1)=2\)
	\end{tabular}
\end{center}

For instance, after processing \(x_4\), the entry \(D'(4,1)=3\) can be obtained by keeping \(x_1\) pending and assigning value \(2\) to \(x_4\). This entry has a smaller weight than \(D'(4,p_\infty)=4\), but it does not represent a complete solution, because the pending vertex \(x_1\) is still undominated. The later vertices \(y_2\) and \(y_3\) cannot dominate \(x_1\), since their intervals start at \(x_2\) and \(x_3\), respectively. Hence, this branch cannot produce an accepted final state of weight \(3\).

At the end, only states of the form \((s,p_\infty)\) are accepted, because no vertex of \(X\) may remain pending after all vertices have been processed. In this example, the minimum accepted value is \(D(4,p_\infty)=4\). Therefore, the algorithm returns \(\gamma_R(G)=4\). The corresponding Roman dominating function is the one shown in Figure~\ref{fig:roman-convex-example}: \(f(y_1)=2\), \(f(x_4)=2\), and all other vertices receive value \(0\).

%% file: RD_Chordal_Bipartite.tex
\section{NP-completeness on Chordal Bipartite Graphs}
\label{sec:NPC_ChordalBG}
In this section, we prove that the Roman domination problem remains hard even when the input graph is restricted to chordal bipartite graphs. We reduce from \textsc{Dominating Set}, which is NP-complete on chordal bipartite graphs~\cite{muller1987np}.

\begin{theorem}
	The decision version of \textsc{Roman Domination} is NP-complete on chordal bipartite graphs.
\end{theorem}

\begin{proof}
	First, the problem belongs to NP. Given a graph \(G\), an integer \(k\), and a function \(f:V(G)\rightarrow\{0,1,2\}\), one can check in polynomial time whether \(w(f)\leq k\) and whether every vertex assigned value \(0\) has a neighbour assigned value \(2\). Thus, \textsc{Roman Domination} belongs to NP.
	
	We prove NP-hardness by a polynomial reduction from \textsc{Dominating Set} on chordal bipartite graphs. Let \((G,k)\) be an instance of \textsc{Dominating Set}, where \(G\) is a chordal bipartite graph. Let \(V(G)=\{v_1,v_2,\ldots,v_n\}\). We construct a graph \(H\) from \(G\) by adding one pendant vertex \(v_i'\) adjacent only to \(v_i\), for every \(i\in\{1,\ldots,n\}\). Thus, \(V(H)=V(G)\cup\{v_1',v_2',\ldots,v_n'\}\) and \(E(H)=E(G)\cup\{v_iv_i':1\leq i\leq n\}\). The construction is illustrated in Figure~\ref{fig:cb-roman-reduction}.
	
	\begin{figure}[h]
		\centering
		\resizebox{0.90\linewidth}{!}{%
		\begin{tikzpicture}[
			vtx/.style={circle, draw, minimum size=7mm, inner sep=0pt},
			leaf/.style={circle, draw, fill=gray!20, minimum size=7mm, inner sep=0pt},
			every node/.style={font=\small}
			]
			
			\node at (1.5,2.2) {\(G\)};
			
			\node[vtx] (g1) at (0,1.2) {\(v_1\)};
			\node[vtx] (g2) at (1,0.4) {\(v_2\)};
			\node[vtx] (g3) at (2,1.2) {\(v_3\)};
			\node[vtx] (g4) at (3,0.4) {\(v_n\)};
			
			\draw (g1)--(g2);
			\draw (g2)--(g3);
			\draw[dashed] (g3)--(g4);
			
			\node at (1.5,-0.5) {Original graph};
			
			\draw[->, thick] (4.0,0.8) -- (5.4,0.8);
			\node at (4.7,1.25) {attach leaves};
			
			\node at (7.5,3.0) {\(H\)};
			
			\node[vtx] (h1) at (6.4,1.2) {\(v_1\)};
			\node[vtx] (h2) at (7.4,0.4) {\(v_2\)};
			\node[vtx] (h3) at (8.4,1.2) {\(v_3\)};
			\node[vtx] (h4) at (9.4,0.4) {\(v_n\)};
			
			\draw (h1)--(h2);
			\draw (h2)--(h3);
			\draw[dashed] (h3)--(h4);
			
			\node[leaf] (h1p) at (6.4,2.0) {\(v_1'\)};
			\node[leaf] (h2p) at (7.4,-0.4) {\(v_2'\)};
			\node[leaf] (h3p) at (8.4,2.0) {\(v_3'\)};
			\node[leaf] (h4p) at (9.4,-0.4) {\(v_n'\)};
			
			\draw (h1)--(h1p);
			\draw (h2)--(h2p);
			\draw (h3)--(h3p);
			\draw (h4)--(h4p);
			
			\node at (8.0,-0.95) {One pendant vertex is added to each \(v_i\)};
			
		\end{tikzpicture}
		}
		\caption{Reduction from \textsc{Dominating Set} to \textsc{Roman Domination}. From the chordal bipartite graph \(G\), we construct \(H\) by attaching one pendant vertex \(v_i'\) to each vertex \(v_i\in V(G)\).}
		\label{fig:cb-roman-reduction}
	\end{figure}
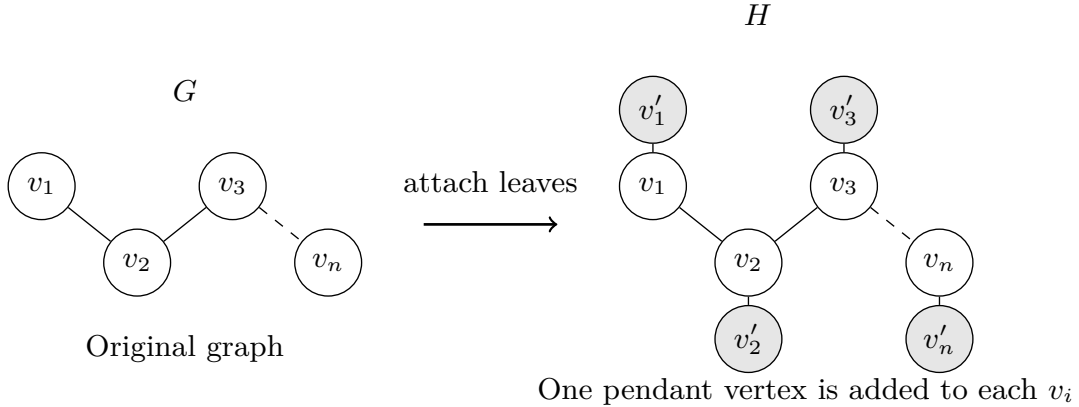
	
	We first observe that \(H\) is chordal bipartite. If \(G\) has a bipartition \((A,B)\), then each pendant vertex attached to a vertex of \(A\) is placed in \(B\), and each pendant vertex attached to a vertex of \(B\) is placed in \(A\). Hence, \(H\) is bipartite. Moreover, adding pendant vertices creates no new cycle. Therefore, every induced cycle of \(H\) is already an induced cycle of \(G\). Since \(G\) is chordal bipartite, it has no induced cycle of length at least \(6\). Thus, \(H\) is also chordal bipartite.
	
	We claim that \(\gamma_R(H)=n+\gamma(G)\), where \(\gamma(G)\) is the domination number of \(G\). First, let \(D\) be a dominating set of \(G\). We define \(f:V(H)\rightarrow\{0,1,2\}\) as follows. For each \(v_i\in D\), set \(f(v_i)=2\) and \(f(v_i')=0\). For each \(v_i\notin D\), set \(f(v_i)=0\) and \(f(v_i')=1\). The weight of \(f\) is \(2|D|+(n-|D|)=n+|D|\).
	
	We check that \(f\) is a Roman dominating function of \(H\). If \(v_i'\) has value \(0\), then \(v_i\in D\), and hence \(f(v_i)=2\). Thus, \(v_i'\) is dominated. If \(v_i\notin D\), then \(f(v_i)=0\). Since \(D\) is a dominating set of \(G\), the vertex \(v_i\) has a neighbor \(v_j\in D\), and \(f(v_j)=2\). Thus, \(v_i\) is dominated. Therefore, every vertex assigned value \(0\) has a neighbour assigned value \(2\), and so \(f\) is a Roman dominating function of \(H\). Hence, \(\gamma_R(H)\leq n+\gamma(G)\).
	
	Conversely, let \(f\) be a minimum Roman dominating function of \(H\). We may assume that no pendant vertex \(v_i'\) has value \(2\). Indeed, suppose that \(f(v_i')=2\) for some \(i\). We replace the values on the pair \(\{v_i,v_i'\}\) by setting \(f(v_i)=2\) and \(f(v_i')=0\). This change does not increase the weight. It also preserves the Roman domination property. The only vertex that could have used \(v_i'\) as a value-\(2\) neighbor is \(v_i\), and after the replacement \(v_i\) itself has value \(2\), while \(v_i'\) is dominated by \(v_i\). Applying this replacement whenever necessary, we obtain a minimum Roman dominating function in which no pendant vertex has value \(2\).
	
	Define \(D=\{v_i\in V(G):f(v_i)>0\}\). We show that \(D\) is a dominating set of \(G\). Let \(v_i\notin D\). Then \(f(v_i)=0\). Since no pendant vertex has value \(2\), the pendant vertex \(v_i'\) cannot dominate \(v_i\). Therefore, because \(f\) is a Roman dominating function of \(H\), the vertex \(v_i\) must have a neighbor \(v_j\in V(G)\) with \(f(v_j)=2\). In particular, \(v_j\in D\). Hence, every vertex outside \(D\) is adjacent to a vertex of \(D\), and therefore, \(D\) is a dominating set of \(G\).
	
	It remains to compare the weights. For each \(i\), the pair \(\{v_i,v_i'\}\) contributes at least one unit of weight. Indeed, if both \(v_i\) and \(v_i'\) had value \(0\), then \(v_i'\) would not be dominated, because its only neighbour is \(v_i\). Moreover, if \(v_i\in D\), then \(f(v_i)>0\), and the pair \(\{v_i,v_i'\}\) contributes at least two units of weight. If \(f(v_i)=2\), this is immediate. If \(f(v_i)=1\), then \(v_i'\) cannot have value \(0\), because its only neighbor \(v_i\) has value \(1\), not value \(2\). Hence, \(v_i'\) must have value at least \(1\). Thus, every pair contributes at least one unit, and each pair corresponding to a vertex of \(D\) contributes at least one additional unit. Therefore, \(w(f)\geq n+|D|\). Since \(D\) is a dominating set of \(G\), we have \(|D|\geq \gamma(G)\). Hence, \(w(f)\geq n+\gamma(G)\). Thus, \(\gamma_R(H)\geq n+\gamma(G)\).
	
	Combining both inequalities, we obtain \(\gamma_R(H)=n+\gamma(G)\). Therefore, \(G\) has a dominating set of size at most \(k\) if and only if \(H\) has a Roman dominating function of weight at most \(n+k\). The construction is polynomial and preserves chordal bipartiteness. Hence, the decision version of \textsc{Roman Domination} is NP-hard on chordal bipartite graphs. Since the problem also belongs to NP, it is NP-complete on chordal bipartite graphs.
\end{proof}